\documentclass{amsart}
\usepackage{graphicx}
\usepackage{amssymb,amsmath,amsfonts, amsthm}
\usepackage{amscd,enumerate}
\usepackage{times}
\newtheorem{theorem}{Theorem}[section]
\newtheorem{proposition}[theorem]{Proposition}

\newtheorem{conjecture}[theorem]{Conjecture}
\newtheorem{question}[theorem]{Question}

\newtheorem{corollary}[theorem]{Corollary}

\theoremstyle{remark}

\theoremstyle{definition}

\newcommand{\Q}{\Bbb Q}

\newcommand{\G}{G(K)}
\newcommand{\C}{\Bbb C}
\newcommand{\Z}{\Bbb Z}
\newcommand{\D}{\Delta}

\newcommand{\tr}{{\mathrm{tr}\,}}

\numberwithin{equation}{section}

\begin{document}

\title[Twisted Alexander polynomials on curves]
{Twisted Alexander polynomials on curves in character varieties of knot groups}

\author[T. Kim, T. Kitayama and T. Morifuji]
{Taehee Kim, Takahiro Kitayama and Takayuki Morifuji}

\begin{abstract}
For a fibered knot in the $3$-sphere
the twisted Alexander polynomial associated to an $SL(2, \C)$-character
is known to be monic.
It is conjectured that for a nonfibered knot there is a curve component
of the $SL(2, \C)$-character variety containing only finitely many characters
whose twisted Alexander polynomials are monic, i.e. finiteness
of such characters detects fiberedness of knots.
In this paper we discuss the existence of a certain curve component
which relates to the conjecture when knots have nonmonic Alexander polynomials.
We also discuss the similar problem of detecting the knot genus.
\end{abstract}



\address{Department of Mathematics, Konkuk University, Seoul 143-701,
Republic of Korea}
\email{tkim@konkuk.ac.kr}

\address{Graduate School of Mathematical Sciences, the University of Tokyo,
3-8-1 Komaba, Meguro-ku, Tokyo 153-8914, Japan}
\email{kitayama@ms.u-tokyo.ac.jp}

\address{Department of Mathematics, Hiyoshi Campus, Keio University,
Yokohama 223-8521, Japan}
\email{morifuji@z8.keio.jp}

\maketitle


\section{Introduction}

The twisted Alexander polynomial was introduced by Lin~\cite{Lin01-1}
for knots in the $3$-sphere and by Wada~\cite{Wada94-1}
for finitely presentable groups.
It is a generalization of the Alexander polynomial and
gives a powerful tool in low dimensional topology.
One of the most notable applications is detecting
fibered knots or more generally fibered $3$-manifolds.
To be more precise,
Friedl and Vidussi showed in \cite{FV11-1} that
the twisted Alexander polynomials associated to
finite representations determine whether knot complements
and general irreducible $3$-manifolds are fibered over the circle.
Another important application is detecting the genus of knots.
More generally,
Friedl and Vidussi showed that
the twisted Alexander polynomials associated to
finite representations also determine the Thurston norms of
irreducible $3$-manifolds which are not closed graph manifolds \cite{FV12-1}.
For literature on the twisted Alexander polynomial and
other related topics, we refer to the survey paper by
Friedl and Vidussi \cite{FV10-1}.

In this paper
we study the problems of detecting fiberedness and
the genus $g(K)$ of a knot $K$ by twisted Alexander polynomials
from the viewpoint of the $SL(2,\C)$-character variety of a knot group.
In this point of view,
we consider the regular functions on the character variety induced by the coefficients of the
twisted Alexander polynomials associated to characters of a knot group. In particular,
the regular function induced by the coefficients of the highest degree terms turns out to contain
much information of a knot.
We call a representation and its character monic
if the highest coefficient of the associated twisted Alexander polynomial is one
(see \cite{GM03-1} for example).
Moreover, we say that a representation and its character determines the knot genus
if the degree of the associated twisted Alexander polynomial equals $4g(K)-2$.
Using these terminologies,
we can say
that every $SL(2,\C)$-representation
(and its character) of a fibered knot is monic \cite{GKM05-1}
and determines the knot genus \cite{KM05-1}.

It is natural to ask whether the converse is true.
More precisely, one can ask if every $SL(2,\C)$-character is monic for a knot,
then the knot is fibered.
Regarding detecting the knot genus, a natural question also arises:
for a (possibly nonfibered) knot, does there exist
an irreducible $SL(2,\C)$-character which determines the knot genus?
However, only a few partial answers are known so far (see \cite{Morifuji08-1} for twist knots and
\cite{KM10-1} for $2$-bridge knots).
In fact, for 2-bridge knots it is shown that certain finiteness properties
of a curve component in the character variety detect fiberedness and the genus \cite{KM10-1}.
More generally,
we conjecture that for a nontrivial knot
there is an irreducible component in the $SL(2,\C)$-character variety
which satisfies a certain finiteness condition (see Questions~\ref{question:fibering} and \ref{question:genus}).

The purpose of the present paper is to give
some evidence that the conjecture is true
for a wide class of knots with nonmonic Alexander polynomials.
More generally,
we give several sufficient conditions which ensure the existence
of a certain curve component in the character variety
which relates to the conjecture mentioned above (see Sections 3 and 4).
For instance, in Theorem~\ref{thm:simple-root} we show that if a knot $K$ has
the Alexander polynomial $\D_K(t)$ which is nonmonic and has a simple root,
then there is a
curve component of the $SL(2,\C)$-character variety 
of the knot group of $K$ which contains the character
of an irreducible representation and has only finitely many monic characters.

Our criteria are sufficiently applicable and we can show
the existence of such curves for all nonfibered prime knots with $10$
or fewer crossings.
The results stated in this paper
use information of the Alexander polynomial,
however
there seems to be no a priori relation between
finiteness properties of a curve component and the Alexander polynomial.

This paper is organized as follows.
In Section 2,
we quickly review some basic materials of the character variety
and the twisted Alexander polynomials
associated to $SL(2,\C)$-representations.
Here we also recall a conjecture of Dunfield, Friedl and Jackson \cite{DFJ11-1}
on the twisted Alexander polynomial of the holonomy representation
for hyperbolic knots.
In Section 3,
we show finiteness of monic characters in curve components
of the character varieties for a wide class of nonfibered knots.
In particular, we show that
our method can also be applied to satellite knots (hence nonhyperbolic knots).
In Section 4,
we apply the argument in Section 3 to the similar problem of
detecting the knot genus.
In Section 5, as an example,
we give explicit computations of curves in the character variety and
count the number of monic characters for a Montesinos knot
with length $3$.

\section{Preliminaries}

\subsection{Character variety}

In this subsection,
we review some basics on character varieties following \cite{CS83-1}.
Let $G$ be a finitely generated group.
The {\it variety of representations} of $G$ is the set of
$SL(2,\C)$-representations:
$R(G)=\mathrm{Hom}(G,SL(2,\C))$.
Since $G$ is finitely generated, $R(G)$ can be embedded in
a product $SL(2,\C)\times\cdots\times SL(2,\C)$ by mapping
each representation to the image of a generating set.
In this manner,
$R(G)$ is an affine algebraic set whose defining polynomials
are induced by the relations of a presentation of $G$.
It is not hard to see
that this structure is independent of the choice of
presentations of $G$ up to isomorphism.

A representation $\rho \colon G\to SL(2,\C)$ is said to be
{\it abelian} if $\rho(G)$ is an abelian subgroup of
$SL(2,\C)$. A representation $\rho$ is called {\it reducible}
if there exists a proper invariant subspace in $\C^2$
under the action of $\rho(G)$.
This is equivalent to saying that $\rho$ can be conjugated to a
representation by upper triangular matrices. It is easy to see that
every abelian representation is reducible, but the converse does not
hold. Namely there is a reducible nonabelian representation in
general. When $\rho$ is not reducible, it is called
{\it irreducible}.
We denote the subset of $R(G)$
consisting of irreducible $SL(2, \C)$-representations by $R^{\text{irr}}(G)$.

Given a representation $\rho\in R(G)$, its {\it character} is
the map $\chi_\rho \colon G\to \C$ defined by
$\chi_\rho(\gamma)=\tr(\rho(\gamma))$ for $\gamma\in G$.
We denote the set of all characters by $X(G)$.
For a given element $\gamma\in G$,
we define the map $\tau_\gamma:X(G)\to\C$ by
$\tau_\gamma(\chi)=\chi(\gamma)$.
Then
it is known that $X(G)$ is an affine algebraic set
which embeds in $\C^N$ with coordinates
$(\tau_{\gamma_1},\ldots,\tau_{\gamma_N})$
for some $\gamma_1,\ldots,\gamma_N\in G$.
This affine algebraic set is called the {\it character variety}
of $G$.
We note that the set $\{\gamma_1,\ldots,\gamma_N\}$ can
be chosen to contain a generating set of $G$.
The projection $t \colon R(G)\to X(G)$ given by
$t(\rho)= \chi_\rho$ is surjective.
We denote the Zariski closure of $t(R^{\text{irr}}(G))$ by $X^{\text{irr}}(G)$.

Let $E_K=S^3\backslash\text{int}(N(K))$,
the exterior of a knot $K$ in the 3-sphere.
For a knot group $G(K)=\pi_1 E_K$,
we write $R(K)=R(G(K))$, $R^{\text{irr}}(K)=R^{\text{irr}}(G(K))$, $X(K)=X(G(K))$
and $X^{\text{irr}}(K)=X^{\text{irr}}(G(K))$ for simplicity.

Let $\eta_\lambda \colon G(K)\to SL(2,\C)$ be the abelian representation
given by
$\mu\mapsto \begin{pmatrix}\lambda&0\\0&\lambda^{-1}\end{pmatrix}$,
where $\mu$ denotes a meridian of $K$.
By results of Burde \cite{Burde67-1}
and de Rham \cite{deRham67-1},
there is a reducible nonabelian representation
$\rho_\lambda \colon G(K)\to SL(2,\C)$ such that
$\chi_{\eta_\lambda}=\chi_{\rho_\lambda}$
if and only if
$\lambda^2$ is a root of the Alexander polynomial $\D_K(t)$.

Here let us recall the following results,
due to Heusener, Porti and Su\'{a}rez Peir\'{o}
\cite{HPS01-1} (see also Shors \cite{Shors91-1}),
and due to Herald \cite{Herald97-1} and Heusener-Kroll \cite{HK98-1},
on the local description of a reducible character in $X^{\text{irr}}(K)$.
We denote by $Y(K)$ the curve component of $X(K)$
consisting of abelian characters.

\begin{proposition}{\cite[Theorem~1.2]{HPS01-1}} \label{pro:HPS}
If $\lambda^2$ is a simple root of $\D_K(t)$,
there is a unique irreducible curve component $X_0$ of $X^{\text{irr}}(K)$
such that $\chi_{\rho_\lambda}\in X_0 \cap Y(K)$.
\end{proposition}

The following proposition is an immediate consequence of
\cite[Theorem~1]{Herald97-1} or \cite[Theorem~1.1]{HK98-1}.

\begin{proposition}\label{pro:HHK}
If the equivariant knot signature function $\sigma_K \colon U(1) \to \Z$ changes
its value at $e^{2 i \alpha}$ $(\alpha \in [0, \pi])$, then there is an irreducible component $X_0$
of $X^{\text{irr}}(K)$ such that $\chi_{\rho_{e^{i \alpha}}} \in X_0 \cap Y(K)$.
\end{proposition}

The equivariant knot signature function $\sigma_K$ is also called
the Levine-Tristram signature function (for example, see \cite[Section 2.1]{HK98-1} for the details).
We note that in this paper the signature function $\sigma_K$ is considered to be the 
{\it averaged signature function}.
Namely, for $\omega\in U(1)=S^1$ the value $\sigma_K(\omega)$ is redefined to be
the limit of the average of the values $\sigma_K(\omega_+)$ and $\sigma_K(\omega_-)$
where $\omega_+$ and $\omega_-$ are points on $S^1$ approaching $\omega$ from opposite sides.
It is known that $\sigma_K(1)=0$, and the function $\sigma_K$ is locally constant except
at zeros of $\D_K(t)$. It is also known that if $e^{2 i \alpha}$ is an odd multiple root of $\D_K(t)$,
then $\sigma_K$ changes its value at $e^{2 i \alpha}$.

\subsection{Twisted Alexander polynomials}

Following Wada \cite{Wada94-1}, we define the twisted Alexander polynomials.
First we fix an epimorphism
$\alpha \colon G(K) \to \langle t \rangle$ and a Wirtinger presentation
\[ \G = \langle \gamma_1,\ldots, \gamma_n\,|\,r_1,\ldots, r_{n-1}\rangle. \]
For a given representation $\rho\colon \G\to GL(2,\C)$, we can extend
the group homomorphism
$\alpha \otimes \rho \colon G(K) \to GL(2, \C[t^{\pm 1}])$ to
a ring homomorphism $\Phi \colon \Z[G(K)] \to M(2, \C(t))$.

We consider the $(n-1) \times n$ matrix $M$ whose $(i, j)$-entry is
$\frac{\partial r_i}{\partial \gamma_j} \in \Z[G(K)]$, where
$\frac{\partial}{\partial \gamma_j}$ denotes the Fox differential.
For $1 \leq k \leq n$, we denote by $M_k$ the $(n-1) \times (n-1)$ matrix
obtained from $M$ by removing the $k$th column, and by $\Phi(M_k)$ the matrix
in $M(2(n-1), \C(t))$ obtained by taking the images of entries of $M_k$ by $\Phi$.
Then the {\it twisted Alexander polynomial $\D_{K,\rho}(t) \in \C(t)$
associated with $\rho \colon G(K)\to GL(2,\C)$}
is defined as
\[ \D_{K,\rho}(t) = \frac{\det \Phi(M_k)}{\det \Phi(\gamma_k -1)}, \]
which is well-defined up to multiplication by
$\epsilon t^{2i}$ $(\epsilon\in \C^*,\,i\in \Z)$.
In the case that $\rho$ is a nonabelian special linear representation
$\rho\colon \G\to SL(2,\C)$, we have $\D_{K,\rho}(t)\in
\C[t^{\pm 1}]$ \cite[Theorem 3.1]{KM05-1} and
it is well-defined up to multiplication by $t^{2i}\,(i\in \Z)$.
We note that if $\rho$ and $\eta$ are mutually conjugate
$SL(2,\C)$-representations,
then $\D_{K,\rho}(t) = \D_{K,\eta}(t)$ holds.
If $\rho$ and $\eta\colon \G \to SL(2,\C)$ are irreducible
representations with $\chi_\rho = \chi_\eta$, then $\rho$ is conjugate to $\eta$
(see \cite[Proposition 1.5.2]{CS83-1}), and hence
$\D_{K,\rho}(t) = \D_{K,\eta}(t)$.
And if $\rho$ and $\eta$ are reducible,
then they are determined by $\D_K(t)$ and hence
$\D_{K,\rho}(t) = \D_{K,\eta}(t)$
(see the proof of \cite[Theorem 3.1]{KM05-1}).
Therefore, we can define the twisted Alexander polynomial
associated with $\chi \in X(K)$ to be $\D_{K,\rho}(t)$
where $\chi = \chi_\rho$ and we denote it by $\D_{K,\chi}(t)$.

\subsection{A conjecture of Dunfield, Friedl and Jackson}

We say a nonabelian representation $\rho\colon \G\to SL(2,\C)$
(resp. a nonabelian character $\chi$) is {\it monic} if $\D_{K,\rho}(t)$
(resp. $\D_{K,\chi}(t)$) is a monic polynomial,
that is, its coefficient of the highest degree term is $1$.
Note that we do not allow $-1$ as the coefficient for monicness in this paper.
It is well-known that every nonabelian representation of a fibered knot is monic \cite[Theorem~3.1]{GKM05-1},
and therefore so is every nonabelian character.

It is also well-known that the degrees of twisted Alexander polynomials
give genus bounds from below \cite{FK06-1}.
In particular, in our settings we have
\[ 4g(K)-2 \geq \deg \D_{K, \chi}(t), \]
for every $\chi \in X(K)$.
If the equality holds, then we say that $\chi$ determines the genus.
It is known that
for any $\chi \in X(K)$ of a fibered knot,
$\chi$ determines the genus \cite[Theorem~3.2]{KM05-1}.

For a hyperbolic knot $K$, the holonomy representation
$\bar{\rho}_0 \colon G(K) \to PSL(2, \C)$ has a lift $\rho_0 \colon G(K) \to SL(2, \C)$,
see \cite[Proposition 3.1.1]{CS83-1} for the detail.
Dunfield, Friedl and Jackson \cite{DFJ11-1} presented the following conjecture,
and confirmed it for all hyperbolic knots of 15 or fewer crossings.
See also \cite[Conjecture~1.13]{DFJ11-1}.

\begin{conjecture}{\cite[Conjecture~1.7]{DFJ11-1}}\label{conj:dunfield-friedl-jackson}
Let $K$ be a hyperbolic knot
and $\rho_0 \colon G(K) \to SL(2, \C)$ a lift of the holonomy representation.
Then $K$ is fibered if and only if $\chi_{\rho_0}$ is monic.
Moreover, $\chi_{\rho_0}$ determines the knot genus $g(K)$.
\end{conjecture}

Recently, the third author confirmed the conjecture for twist knots
\cite{Morifuji12-1}, and the third author and Tran did for a certain wider
class of $2$-bridge knots \cite{MT13-1}.
These are the first infinite families of knots where
Conjecture~\ref{conj:dunfield-friedl-jackson} is verified.

It is known that we can write the twisted Alexander polynomial $\D_{K,\chi}(t)$ without any ambiguity as
\[
\D_{K,\chi}(t)
=
\sum_{j=0}^{4g-2}\psi_j(\chi) t^j.
\]
with $\C$-valued functions $\psi_j$ on $X^{\text{irr}}(K)$ such that $\psi_k=\psi_{4g-2-k}~(0\leq k\leq 2g-1)$
where $g=g(K)$ (see \cite[Theorem 1.5]{FKK11-1} and its proof). 
Then for a subvariety $X_0$ of $X^{\text{irr}}(K)$ we say that $\psi_n$ is
\emph{the coefficient of the highest degree term of $\D_{K,\chi}(t)$ on $X_0$}
if $\psi_m\equiv 0$ for $m>n$ and $\psi_n\not\equiv 0$ on $X_0$.
We end this section with the following useful proposition.

\begin{proposition}\label{pro:finitely-many-points}
Let $K$ be a knot of genus $g$
and $\chi$ a character in a curve component
$X_0$ of $X^{\text{irr}}(K)$. We write the twisted Alexander polynomial
$\D_{K,\chi}(t)$ as above:
\[
\D_{K,\chi}(t)
=
\sum_{j=0}^{4g-2}\psi_j(\chi) t^j.
\]
Then each $\psi_k$ is a regular function on $X_0$, and
if $\psi_k\not\equiv c$ on $X_0$ for a constant $c\in \C$,
then
$\psi_k=c$ for finitely many points of $X_0$.
\end{proposition}

\begin{proof}
It follows from \cite[Theorem~1.5]{DFJ11-1} that $\psi_k$ is a regular function.
Therefore $\psi_k^{-1}(c)$ is a subvariety of the curve $X_0$ of codimension one.
In particular, it consists of finitely many points. (Also see \cite[Corollary~1.6]{DFJ11-1}.)
\end{proof}

\section{Fibering and monic characters}

In this section
we present some classes of nonfibered knots for which the following question 
is affirmatively solved, in particular, on a curve component.

\begin{question}\label{question:fibering}
For a nonfibered knot $K$, is there an irreducible component of
$X^{\text{irr}}(K)$ with finitely many monic characters?
\end{question}

Combined with Proposition~\ref{pro:finitely-many-points},
Conjecture~\ref{conj:dunfield-friedl-jackson} implies an affirmative answer
to Question~\ref{question:fibering}
for hyperbolic knots
since the characters of lifts of the holonomy representations are known to be contained
in unique curve component.
This unique curve component of a hyperbolic knot is called the {\it canonical component}.
The third author solved the question affirmatively for twist knots \cite{Morifuji08-1},
and later the first and third authors did for $2$-bridge knots \cite{KM10-1} in a strong sense.
More precisely, the following theorem holds from
\cite[Theorem~4.3,~Remark~2.3]{KM10-1}.
Let $X_0\subset X^{\text{irr}}(K)$ be an irreducible component.
We say $X_0$ satisfies
\emph{Property $(F)$} if $X_0$ contains
finitely many monic characters and an abelian character.

\begin{theorem}{\cite[Theorem~4.3]{KM10-1}}\label{thm:kim-morifuji-fibered}
For a nonfibered $2$-bridge knot, there is a curve component of $X^{\text{irr}}(K)$
satisfying Property $(F)$.
\end{theorem}

At the present,
except some special cases,
we do not know if the curve component appeared in Theorem~\ref{thm:kim-morifuji-fibered}
is the canonical one (see \cite[Theorem~4.6,~Remark~6.5]{KM10-1}).

Our first theorem in this paper is the following.

\begin{theorem}\label{thm:simple-root}
Let $K$ be a knot such that $\D_K(t)$ is nonmonic and has a simple root.
Then there is a curve component of $X^{\text{irr}}(K)$ satisfying
Property $(F)$.
\end{theorem}

\begin{proof}
By Proposition~\ref{pro:HPS}, there is a reducible character
$\chi_{\rho_\lambda}$ and an irreducible curve component
$X_\lambda$ of $X^{\text{irr}}(K)$ such that $\D_K(\lambda^2)=0$
and 
$\chi_{\rho_\lambda}\in X_\lambda\cap Y(K)$. 
Since the twisted Alexander polynomial associated with
$\rho_\lambda$ is given by
\[
\D_{K,\rho_\lambda}(t)
=
\frac{\D_K(\lambda t)\D_K(\lambda^{-1}t)}{(t-\lambda)(t-\lambda^{-1})}
\]
(see \cite[Remark~3.1~(iv)]{KM10-1}),
$\D_{K,\rho_\lambda}(t)$ is monic if and only if
$\D_K(t)$ is monic.
Hence,
by the assumption,
the coefficient of the highest degree term 
of the twisted Alexander polynomial on $X_\lambda$ is
not the constant one and
the assertion follows from
Proposition~\ref{pro:finitely-many-points}.
This completes the proof.
\end{proof}

As an immediate corollary,
if $\D_K(t)$ is irreducible over $\Q$ and nonmonic,
the knot $K$ satisfies the assumption of Theorem~\ref{thm:simple-root}.
It is known that a prime knot $K$ of $10$ or fewer crossings is fibered
if and only if $\D_K(t)$ is monic.
Then it can be checked that for all non $2$-bridge and nonfibered prime knots
with $10$ or fewer crossings their Alexander polynomials have a simple root,
although they might have nontrivial multiple factors in $\D_K(t)$.
Hence, by Theorems~\ref{thm:kim-morifuji-fibered} and \ref{thm:simple-root},
nonfibered prime knots with $10$ or fewer crossings have curve components
which satisfy Property $(F)$.

Theorem~\ref{thm:simple-root} can also be applied to satellite knots,
and it shows that Question~\ref{question:fibering} also makes sense for nonhyperbolic knots.
Let $\tilde{K}$ be a knot embedded in a standard solid torus
$\tilde{V}=S^1\times D^2\subset S^3$.
We assume that $\tilde{K}$ is not isotopic to $S^1\times\{0\}$
nor is contained in any $3$-ball in $\tilde{V}$.
Let $h$ be a homeomorphism from $\tilde{V}$ onto a closed tubular
neighborhood of a nontrivial knot $\hat{K}$ which maps a longitude
of $\tilde{V}$ onto a longitude of $\hat{K}$.
The image $K=h(\tilde{K})$ is called a \textit{satellite knot} with
\textit{companion knot} $\hat{K}$ and \textit{pattern} $(\tilde{V},\tilde{K})$.
The \textit{winding number} of $\tilde{K}$ in $\tilde{V}$ is the
nonnegative integer $n$ such that the homomorphism
$H_1\tilde{K}\to H_1\tilde{V}\cong \Z$ induced by
the inclusion has the image $n\Z$.
Under the notations above, it is known that the Alexander polynomial of
a satellite knot $K$, with pattern $\tilde{K}$, companion $\hat{K}$ and the winding number $n$
satisfies the following:
\[ \D_K(t)=\D_{\tilde{K}}(t)\cdot\D_{\hat{K}}(t^n). \]
Hence, by Theorem~\ref{thm:simple-root},
for a satellite knot $K$ with pattern $\tilde{K}$ and the \textit{winding number zero}\/
such that $\D_{\tilde{K}}(t)$ is nonmonic and has a simple root,
there is a curve component satisfying Property $(F)$.

Recall that a knot $K$ is called {\it small}
if the exterior $E_K$ contains no closed embedded essential surface.
It is known that torus knots~\cite{Jaco80-1}, $2$-bridge knots~\cite{GL84-1, HT85-1}
and Montesinos knots with length $3$~\cite[Corollary 4]{Oertel84-1} are small.
It is also known that some knots of braid index $3$ or $4$ are small
(see \cite{Finkelstein98-1} and \cite{Matsuda01-1}).

\begin{theorem}\label{thm:signature}
Let $K$ be a small knot such that $\D_K(t)$ is nonmonic.
If the equivariant knot signature function $\sigma_K$ is not identically zero,
then there is a curve component
of $X^{\text{irr}}(K)$ satisfying Property $(F)$.
\end{theorem}

\begin{proof}
Suppose that $\sigma_K$ changes its value at $e^{2 i \alpha}$.
By Proposition~\ref{pro:HHK}, there is a reducible character
$\chi_{\rho_{e^{i \alpha}}}$ and an irreducible component
$X_0$ of $X^{\text{irr}}(K)$ such that $\chi_{\rho_{e^{i \alpha}}} \in X_0 \cap Y(K)$.
Since $K$ is small, $X_0$ is a curve
\cite[Section 2.4]{CCGLS94-1}.
Now the similar argument as in the proof of Theorem~\ref{thm:simple-root}
can work in this setting.
\end{proof}

\begin{theorem} \label{thm:epimorphism}
Let $K$ be a small knot such that $\D_K(t)$ is nonmonic.
If there are a knot $K'$ and an epimorphism $\phi \colon G(K) \to G(K')$
such that there is a component $X_0'$ of $X^{\text{irr}}(K')$
satisfying Property $(F)$, then there is
a curve component $X_0$ of $X^{\text{irr}}(K)$
satisfying Property $(F)$.
\end{theorem}

\begin{proof}
It is straightforward to see that the regular map $\phi^* \colon X(K') \to X(K)$
induced by the epimorphism $\phi$ is injective.
We set $X_0$ to be the image $\phi^*(X_0')$.
It is a curve component of $X^{\text{irr}}(K)$ containing an abelian character
since $K$ is small. Moreover, the composition of an irreducible representation
and an epimorphism of groups is also irreducible.

An abelian character in $X_0$ can be written as $\chi_{\eta_\lambda}$
by an abelian representation $\eta_\lambda$.
Since $\Delta_K(t)$ is not monic, neither is $\D_{K,\chi_{\eta_\lambda}}(t)$.
Hence the coefficient of the highest degree term of $\D_{K,\chi}(t)$ on $X_0$ is not the constant one.
Now the theorem follows from Proposition~\ref{pro:finitely-many-points}.
\end{proof}

The following corollary is an immediate consequence of
Theorems~\ref{thm:kim-morifuji-fibered}, \ref{thm:simple-root},
\ref{thm:signature} and \ref{thm:epimorphism}.

\begin{corollary} \label{cor:epimorphism}
Let $K$ be a small knot.
If there are a knot $K'$ satisfying one of the following:
\begin{enumerate}
\item $K'$ is a nonfibered $2$-bridge knot,
\item $\D_{K'}(t)$ is nonmonic and has a simple root,
\item $\D_{K'}(t)$ is nonmonic and $\sigma_{K'}$ is not identically zero,
\end{enumerate}
and an epimorphism $\phi \colon G(K) \to G(K')$,
then there is a curve component $X_0$ of $X^{\text{irr}}(K)$
which satisfies Property $(F)$.
\end{corollary}

It is well-known that if there is an epimorphism $\phi \colon G(K) \to G(K')$,
then $\D_{K'}(t)$ divides $\D_K(t)$.
Moreover,
it is also known that $\D_{K'}(t)$ is nonmonic for a
nonfibered $2$-bridge knot $K'$.
Hence, in the above theorem,
it turns out that $\D_K(t)$ is nonmonic.

From \cite[Section~8.2]{Kitano11-1}
(see also \cite{IMS11-1}, \cite[Section~9]{ORS08-1}),
we see that for a given $2$-bridge knot $K'$
there exists a Montesinos knot $K$ with length $3$ such that
$G(K)$ admits an epimorphism to $G(K')$.
Namely there are infinitely many small knots
which satisfy the assumption of Corollary~\ref{cor:epimorphism}.

\section{Detecting genus}

In this section we consider the following analogous question
to Question~\ref{question:fibering} on detecting the knot genus.

\begin{question}\label{question:genus}
For a nontrivial knot $K$, is there a component of $X^{\text{irr}}(K)$
where all but finitely many characters determine the knot genus?
\end{question}

As was mentioned in the introduction,
every character of a fibered knot determines the knot genus.
Hence,
Question~\ref{question:genus} has an obvious positive answer for
hyperbolic fibered knots
since the canonical components satisfy the condition.
Moreover
Conjecture~\ref{conj:dunfield-friedl-jackson} implies an affirmative answer
to the question for hyperbolic nonfibered knots.
The first and third authors solved the question affirmatively for
$2$-bridge knots \cite{KM10-1} in a strong sense.
Let $X_0\subset X^{\text{irr}}(K)$ be an irreducible component.
We say $X_0$ satisfies \emph{Property $(G)$}
if all but finitely many characters in $X_0$ determine the knot genus
and $X_0$ contains an abelian character.

\begin{theorem}{\cite[Theorem~4.4]{KM10-1}}\label{thm:kim-morifuji-genus}
For a $2$-bridge knot, there is a curve component of $X^{\text{irr}}(K)$
which satisfies Property $(G)$.
\end{theorem}

As before we do not know if the curve component appeared in
Theorem~\ref{thm:kim-morifuji-genus} is the canonical one.
Furthermore
it is nontrivial whether there is a curve component which satisfies
Property $(G)$ even for a hyperbolic fibered knot.

Analogous arguments for Theorems~\ref{thm:simple-root},
\ref{thm:signature}, \ref{thm:epimorphism} and Corollary~\ref{cor:epimorphism}
present an affirmative answer to Question~\ref{question:genus}
for similar classes of knots.

\begin{theorem}\label{thm:genus-simple-root}
Let $K$ be a knot with $\deg \D_K(t)=2g(K)$ and
$\D_K(t)$ has a simple root.
Then there is a curve component $X_0$ of $X^{\text{irr}}(K)$ which satisfies Property $(G)$.
\end{theorem}

\begin{proof}
As in the proof of Theorem~\ref{thm:simple-root},
by Proposition~\ref{pro:HPS}
there is a curve component $X_\lambda$ of $X^{\text{irr}}(K)$
containing a reducible character $\chi_{\rho_\lambda}$ such that
\[
\D_{K,\chi_{\rho_\lambda}}(t)
=
\frac{\D_K(\lambda t)\D_K(\lambda^{-1}t)}{(t-\lambda)(t-\lambda^{-1})},
\]
which determines the genus by the assumption.
Hence the coefficient
of the highest degree term of $\D_{K,\chi}(t)$ on $X_\lambda$ is not identically zero,
and the theorem follows from Proposition~\ref{pro:finitely-many-points}.
\end{proof}

By Theorems \ref{thm:kim-morifuji-genus}, \ref{thm:genus-simple-root}
and an analogous argument to Property $(F)$ we can check that all prime knots
with $10$ or fewer crossings, except the following seven $3$-bridge knots,
have curve components which satisfy Property $(G)$:
\begin{align*}
&\D_{8_{10}}(t)
=\D_{10_{143}}(t)
=(t^2-t+1)^3\\
&\D_{8_{20}}(t)
=\D_{10_{140}}(t)
=(t^2-t+1)^2\\
&\D_{10_{99}}(t)
=(t^2-t+1)^4\\
&\D_{10_{123}}(t)
=(t^4-3t^3+3t^2-3t+1)^2\\
&\D_{10_{137}}(t)
=(t^2-3t+1)^2.
\end{align*}
Note that it is known that for all prime knots with $10$ or fewer crossings,
$\deg \D_K(t) = 2 g(K)$.

The knots $8_{10}, 8_{20},10_{137}$, $10_{140}$ and $10_{143}$ are known to be Montesinos knots with length $3$,
hence they are small.
On the other hand, the knots $10_{99}$ and $10_{123}$ are not small.
In fact, we can construct essential surfaces in the exteriors as follows:
The knots $10_{99}$ and $10_{123}$ are obtained as the closure of the $3$-braids
\[ \sigma_1^3 \sigma_2^2 \sigma_1^2 \sigma_2^{-3} \sigma_1^{-2} \sigma_2^{-2},~
\sigma_1^{-3} \sigma_2^{-2} \sigma_1^2 \sigma_2^3 \sigma_1^2 \sigma_2^{-2}. \]
Spheres with $3$ holes separating $\sigma_1$ and $\sigma_2$ give tangle decompositions of the knots.
Connecting $2$ of such spheres by $3$ tubes along the strands of the braids, we obtain embedded surfaces of genus $2$, which can be checked to be essential.
This construction is based on \cite[Theorem 3.2]{LP85-1}. 
One can also check that the knots $10_{99}$ and $10_{123}$ are not small in the list given in \cite{BCT12-1}.

\begin{theorem}\label{thm:genus-signature}
Let $K$ be a small knot such that $\deg \D_K(t) = 2g(K)$.
If the equivariant knot signature function
$\sigma_K$ is not identically zero,
then there is a curve component of $X^{\text{irr}}(K)$
which satisfies Property $(G)$.
\end{theorem}

\begin{proof}
Suppose that $\sigma_K$ changes its value at $e^{2 i \alpha}$.
By Proposition~\ref{pro:HHK}, there is a reducible character
$\chi_{\rho_{e^{i \alpha}}}$ and an irreducible component
$X_0$ of $X^{\text{irr}}(K)$ such that $\chi_{\rho_{e^{i \alpha}}} \in X_0 \cap Y(K)$.
Since $K$ is small, $X_0$ is a curve.
Now the same proof as that of Theorem~\ref{thm:genus-simple-root} works.
\end{proof}

The Alexander polynomials $\D_K(t)$ of the knots $8_{10}$ and $10_{143}$ have no simple root, but
we can check that their equivariant knot signature functions $\sigma_K$
are not identically zero. Hence, by Theorem~\ref{thm:genus-signature},
these two knots have curve components which satisfy Property $(G)$.

\begin{theorem} \label{thm:genus-epimorphism}
Let $K$ be a small knot with $\deg \D_K(t) = 2g(K)$.
If there are a knot $K'$ and an epimorphism $\phi \colon G(K) \to G(K')$
such that there is a component $X_0'$ of $X^{\text{irr}}(K')$
satisfying Property $(G)$,
then there is a curve component $X_0$ of $X^{\text{irr}}(K)$
which satisfies Property $(G)$.
\end{theorem}

\begin{proof}
We set $X_0$ to be the image of $X_0'$ of the injection $X(K') \to X(K)$
induced by $\phi$.
Then $X_0$ is a curve component of $X^{\text{irr}}(K)$ containing an abelian character $\chi_{\eta_\lambda}$
as in the proof of Theorem~\ref{thm:epimorphism}.
Since $\deg \Delta_K(t) = 2g(K)$, the character $\chi_{\eta_\lambda}$ determines the genus.
Hence the coefficient of the highest degree term of $\D_{K,\chi}(t)$ on $X_0$ is not identically zero,
which proves the theorem by Proposition~\ref{pro:finitely-many-points}.
\end{proof}

The equivariant knot signature functions $\sigma_K$
of the knots $8_{20},10_{137}$ and $10_{140}$ are identically zero,
but they admit the following epimorphisms to $2$-bridge knot groups
(see \cite[Theorem~1.1]{KS05-1}):
\[
G(8_{20})\twoheadrightarrow G(3_1),~
G(10_{137})\twoheadrightarrow G(4_1),~
G(10_{140})\twoheadrightarrow G(3_1).
\]
Therefore, by Theorems~\ref{thm:kim-morifuji-genus} and \ref{thm:genus-epimorphism},
these three knots have curve components which satisfy Property $(G)$.

More generally,
we obtain the following as an immediate corollary
of Theorems~\ref{thm:kim-morifuji-genus}, \ref{thm:genus-simple-root}, \ref{thm:genus-signature} and
\ref{thm:genus-epimorphism}.

\begin{corollary}
Let $K$ be a small knot with $\deg \D_K(t) = 2g(K)$.
If there are a knot $K'$ satisfying one of the following:
\begin{enumerate}
\item $K'$ is a $2$-bridge knot,
\item $\deg\D_{K'}(t) = 2g(K')$ and $\D_{K'}(t)$ has a simple root,
\item $\deg\D_{K'}(t) = 2g(K')$ and $\sigma_{K'}$ is not identically zero,
\end{enumerate}
and an epimorphism $\phi \colon G(K) \to G(K')$,
then there is a curve component $X_0$ of $X^{\text{irr}}(K)$
which satisfies Property $(G)$.
\end{corollary}

The Alexander polynomials $\D_K(t)$ of the remained knots $10_{99}$ and $10_{123}$ have no simple root
and their equivariant knot signature functions $\sigma_K$ are identically zero.
By \cite[Theorem~1.1]{KS05-1}
the knot group $G(10_{99})$ admits an epimorphism to $G(3_1)$,
but $10_{99}$ is not small.
Moreover
$G(10_{123})$ admits no epimorphism
to the groups of knots of fewer crossings \cite{KS05-1}.
Therefore
we can say nothing about the existence of curve components
which satisfy Property $(G)$ for these knots.

\section{Example}

Let $K$ be the knot $9_{35}$, which is a nonfibered alternating knot of
genus $1$ with $\D_K(t) = 7t^2-13t+7$.
As in Figure~\ref{fig:9_35}, the knot 
$K$ is the $(-3, -3, -3)$ pretzel knot which is a Montesinos knot
with length $3$, and so $K$ is a small knot.
By Theorem~\ref{thm:simple-root} there is a curve component of $X^{\text{irr}}(K)$
with finitely many monic characters,
and by Theorem~\ref{thm:genus-simple-root}
there is also one where all but finitely many characters determine the genus.
Here we explicitly give such curve components.

\begin{figure}[h]
\centering
\includegraphics[width=5cm, clip]{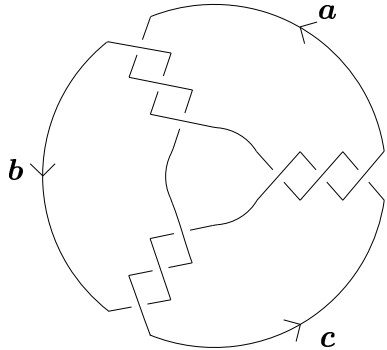}
\caption{The knot $9_{35}$}
\label{fig:9_35}
\end{figure}

The knot $K$ has a period $3$, as is easy to see in Figure~\ref{fig:9_35}.
In fact, $K$ is the inverse image of one unknotted component
of the $2$-bridge link $12/5$ by the $3$-fold branched covering map $S^3 \to S^3$
whose branching set is the other unknotted component.
First, using the method of Hilden, Lozano and Montesinos-Amilibia \cite{HLM95-1, HLM03-1},
we compute defining equations of the curve components of $X(K)$
which come from the character variety of the orbifold fundamental group
of the quotient orbifold by the periodicity.
By \cite[Proposition 5.3]{HLM95-1} the nontrivial components
of the character variety of the link $12/5$ are defined by $r_6(y_1, y_2, v) = 0$,
where $r_6(y_1, y_2, v)$ is inductively defined by
\begin{align*}
r_m(y_1, y_2, v) = &-y_1^{-1}y_2t_{m-1}r_{m-3}(y_2, y_1, v) \\
&+\frac{1}{2}(-2+y_2^2+y_1^{-1}y_2t_{m-1}(2v-y_1y_2))r_{m-2}(y_1, y_2, v) \\
&+\frac{1}{2}(-2y_1^{-1}y_2t_{m-1}+y_1y_2t_{m-1}+2v-y_1y_2)r_{m-1}(y_2, y_1, v), \\
r_0(y_1, y_2, v) = &\,0,~ r_1(y_1, y_2, v) = 1,~ r_2(y_1, y_2, v) = v, \\
t_1 = &\,1,~ t_2 = -1,~ t_3 = 1,~ t_4 = -1,~ t_5 = 1.
\end{align*}
Here $y_1, y_2$ are the trace functions of two standard generators
of the $2$-bridge link group and $v$ is the trace function of the product of these generators.
A computation implies
\[ r_6(y_1, y_2, v) = (v^2-v y_1 y_2+y_1^2+y_2^2-3)(v^3-v^2 y_1 y_2+v y_1^2+v y_2^2-v-y_1 y_2). \]
By \cite[Theorem 3.1]{HLM03-1} it follows from the equation $r_6(y_1, y_2, v) = 0$
that $X(K)$ contains nontrivial curve components defined by
\begin{align*}
f(y, b, w) &=0, \\
(b+2)(wy-b-z)-w^2 &=0, \\
b+1 &=0,
\end{align*}
where $f(y, b, w)$ is the polynomial obtained from $y_1^5 r_6(y_1, y_2, v)$ by the following change of variables:
\begin{align*}
y = y_2, \\
b = y_1^2-2, \\
w = y_1 v.
\end{align*}
By substituting $-1$ for $b$, the equations become
\begin{align*}
(w^2-w y+y^2-2) (w^3-w^2 y+w y^2-y) &=0, \\
w^2-w y+z-1 &=0.
\end{align*}
Taking the resultant in $w$, we have
\[ (y^2-z-1)^2 (y^4 z-2 y^4-2 y^2 z^2+5 y^2 z-2 y^2+z^3-3 z^2+3 z-1) = 0, \]
where $y, z$ are the trace functions of a meridian,
and the product of it and its image by the periodic map, respectively.
We denote by $C, C'$ the curves defined by
\begin{align*}
y^2-z-1 &= 0, \\
y^4 z-2 y^4-2 y^2 z^2+5 y^2 z-2 y^2+z^3-3 z^2+3 z-1 &= 0,
\end{align*}
respectively.

Next we compute the restrictions of the regular function $\psi_2$ to $C, C'$
induced by the highest degree terms of twisted Alexander polynomials
as in Proposition~\ref{pro:finitely-many-points}.
Taking the meridional elements $a, b, c$ depicted in Figure $1$,
we have $G(K) = \langle a, b, c \mid r, s \rangle$, where
\begin{align*}
r &= a\bar{b}ab\bar{a}b\bar{c}b\bar{c}\bar{b}c\bar{b}, \\
s &= b\bar{c}bc\bar{b}c\bar{a}c\bar{a}\bar{c}a\bar{c}.
\end{align*}
Here we write $\bar{a}, \bar{b}, \bar{c}$ for $a^{-1}, b^{-1}, c^{-1}$ respectively.
Note that
\[ y = \tau_a = \tau_b = \tau_c. \]
Since $C, C'$ are symmetric with respect to the periodicity of $X(K)$
induced by that of $K$,
\[ z = \tau_{ab} = \tau_{bc} = \tau_{ca} \]
on these curves.
Since
\begin{align*}
\frac{\partial r}{\partial a} = &1+a\bar{b}-a\bar{b}ab\bar{a}, \\
\frac{\partial r}{\partial b} = &-a\bar{b}+a\bar{b}a+a\bar{b}ab\bar{a}+a\bar{b}ab\bar{a}b\bar{c}-a\bar{b}ab\bar{a}b\bar{c}b\bar{c}\bar{b}-a\bar{b}ab\bar{a}b\bar{c}b\bar{c}\bar{b}c\bar{b}, \\
\frac{\partial s}{\partial a} = &-b\bar{c}bc\bar{b}c\bar{a}-b\bar{c}bc\bar{b}c\bar{a}c\bar{a}+b\bar{c}bc\bar{b}c\bar{a}c\bar{a}\bar{c}, \\
\frac{\partial s}{\partial b} = &1+b\bar{c}-b\bar{c}bc\bar{b},
\end{align*}
for an $SL(2, \C)$-representation $\rho$,
\[ \D_{K, \rho}(t) = \frac{\det (At+B)}{\det (\rho(c) t - I)}, \]
where 
$I$ denotes the $2\times2$ identity matrix and 
\begin{align*}
A &= 
\begin{pmatrix}
-\rho(a\bar{b}ab\bar{a}) & \rho(a\bar{b}a)+\rho(a\bar{b}ab\bar{a})+\rho(a\bar{b}ab\bar{a}b\bar{c}) \\
-\rho(b\bar{c}bc\bar{b}c\bar{a})-\rho(b\bar{c}bc\bar{b}c\bar{a}c\bar{a}) & -\rho(b\bar{c}bc\bar{b})
\end{pmatrix}, \\
B &= 
\begin{pmatrix}
I + \rho(a\bar{b}) & \rho(a\bar{b})-\rho(a\bar{b}ab\bar{a}b\bar{c}b\bar{c}\bar{b})-\rho(a\bar{b}ab\bar{a}b\bar{c}b\bar{c}\bar{b}c\bar{b}) \\
\rho(b\bar{c}bc\bar{b}c\bar{a}c\bar{a}\bar{c}) & I +\rho(b\bar{c})
\end{pmatrix}.
\end{align*}
Hence
\begin{align*}
\psi_2(\chi_\rho) = &\det A \\
= &\det \left(
\begin{smallmatrix}
-\rho(a\bar{b}ab\bar{a}) & \rho(a\bar{b}a)+\rho(a\bar{b}ab\bar{a})+\rho(a\bar{b}ab\bar{a}b\bar{c}) \\
0 & -\rho(b\bar{c}bc\bar{b}c\bar{a})(\rho(c\bar{a}b\bar{c})+\rho(a\bar{b})+\rho(a\bar{c})+\rho(b\bar{c})+\rho(c\bar{a})+\rho(c\bar{b})+I))
\end{smallmatrix}
\right) \\
= &\det (\rho(c\bar{a}b\bar{c})+\rho(a\bar{b})+\rho(a\bar{c})+\rho(b\bar{c})+\rho(c\bar{a})+\rho(c\bar{b})+I) \\
= &\frac{1}{2}((\tr(\rho(c\bar{a}b\bar{c})+\rho(a\bar{b})+\rho(a\bar{c})+\rho(b\bar{c})+\rho(c\bar{a})+\rho(c\bar{b})+I))^2 \\
&-\tr((\rho(c\bar{a}b\bar{c})+\rho(a\bar{b})+\rho(a\bar{c})+\rho(b\bar{c})+\rho(c\bar{a})+\rho(c\bar{b})+I)^2)) \\
= &2(\tr \rho(a\bar{b}))^2+2(\tr \rho(b\bar{c}))^2+2(\tr \rho(c\bar{a}))^2 \\
&+4\tr \rho(a\bar{b}) \tr \rho(b\bar{c})+4\tr \rho(b\bar{c}) \tr \rho(c\bar{a})+4\tr \rho(c\bar{a}) \tr \rho(a\bar{b}) \\
&-\tr \rho(a\bar{b}c\bar{a}b\bar{c})-\tr \rho(a\bar{b}a\bar{b})-\tr \rho(b\bar{c}b\bar{c})-\tr \rho(c\bar{a}c\bar{a}) \\
&-2\tr \rho(a\bar{b}a\bar{c})-2\tr \rho(b\bar{c}b\bar{a})-2\tr \rho(c\bar{a}c\bar{b})-3 \\
= &(2\tau_{a\bar{b}}^2+2\tau_{b\bar{c}}^2+2\tau_{c\bar{a}}^2+4\tau_{a\bar{b}} \tau_{b\bar{c}}+4\tau_{b\bar{c}} \tau_{c\bar{a}}+4\tau_{c\bar{a}} \tau_{a\bar{b}}-\tau_{a\bar{b}c\bar{a}b\bar{c}} \\
&-\tau_{a\bar{b}a\bar{b}}-\tau_{b\bar{c}b\bar{c}}-\tau_{c\bar{a}c\bar{a}}-2\tau_{a\bar{b}a\bar{c}}-2\tau_{b\bar{c}b\bar{a}}-2\tau_{c\bar{a}c\bar{b}}-3)(\chi_\rho).
\end{align*}
Here by trace identities we have
\begin{align*}
\tau_{a\bar{b}} = \tau_{b\bar{c}} = \tau_{c\bar{a}} &= x, \\
\tau_{a\bar{b}a\bar{b}} = \tau_{b\bar{c}b\bar{c}} = \tau_{c\bar{a}c\bar{a}} &= x^2-2, \\
\tau_{a\bar{b}a\bar{c}} = \tau_{b\bar{c}b\bar{a}} = \tau_{c\bar{a}c\bar{b}} &= x^2-x, \\
\tau_{a\bar{b}c\bar{a}b\bar{c}} &= -x^3+3x^2-2
\end{align*}
on $C, C'$, where we set $x = y^2-z$.
Consequently, we obtain
\[ \psi_2 = x^3+6x^2+6x+5 \]
on $C, C'$.

It is easy to check that $\psi_2$ is the constant function with value $18$ on $C$.
In particular, there is no monic character in $C$ and every character in $C$  
determines the knot genus. 
A straightforward computation implies that the number of monic characters in $C'$ is $6$ and
that all but $2$ characters in 
$C'$ determine the genus.

\section*{Acknowledgments}
The authors would like to thank Hiroshi Matsuda for telling us
various constructions of closed essential surfaces in knot complements
and the literature \cite{Finkelstein98-1, LP85-1, Matsuda01-1}.
They would also like to thank the anonymous referee for helpful comments.
The first author was supported by Basic Science Research
Program through the National Research Foundation of Korea (NRF) funded by the
Ministry of Education, Science and Technology (No. 2012R1A1A2012001747 and No. 20120000341).
The second author was supported by JSPS Research Fellowships for Young Scientists.
The third author was partially supported by
Grant-in-Aid for Scientific Research (No.\,23540076),
the Ministry of Education, Culture, Sports, Science
and Technology, Japan.



\begin{thebibliography}{99}
\bibitem{Burde67-1}
G. Burde,
{\it Darstellungen von Knotengruppen},
Math. Ann. {\bf 173} (1967) 24--33.

\bibitem{BCT12-1}
B. Burton, A. Coward and S. Tillmann
{\it Computing closed essential surfaces in knot complements},
arXiv:1212.1531.

\bibitem{CCGLS94-1}
D.~Cooper, M.~Culler, H.~Gillet, D.~D.~Long and P.~B.~Shalen,
{\it Plane curves associated to character varieties of $3$-manifolds},
Invent. Math. {\bf 118} (1994), 47--84.

\bibitem{CS83-1}
M. Culler and P. B. Shalen,
{\it Varieties of group representations and splittings of $3$-manifolds},
Ann. of Math. {\bf 117} (1983) 109--146.

\bibitem{deRham67-1}
G. de Rham,
{\it Introduction aux polynomes d'un noeud},
Enseignement Math. (2) {\bf 13} (1967) 187--194.

\bibitem{DFJ11-1}
N. M. Dunfield, S. Friedl and N. Jackson,
{\it Twisted Alexander polynomials of hyperbolic knots},
Experiment. Math. {\bf 21} (2012), 329--352. 

\bibitem{Finkelstein98-1}
E. Finkelstein,
{\it Closed incompressible surfaces in closed braid complements},
J. Knot Theory Ramifications {\bf 7} (1998), 335--379.

\bibitem{FK06-1}
S. Friedl and T. Kim,
{\it The Thurston norm, fibered manifolds and twisted Alexander polynomials},
Topology {\bf 45} (2006) 929--953.

\bibitem{FKK11-1}
S. Friedl, T. Kim and T. Kitayama,
{\it Poincar\'{e} duality and degrees of twisted Alexander polynomials},
to appear in Indiana Univ. Math. J.

\bibitem{FV10-1}
S. Friedl and S. Vidussi,
{\it A survey of twisted Alexander polynomials},
The Mathematics of Knots: Theory and Application
(Contributions in Mathematical and Computational Sciences),
eds. Markus Banagl and Denis Vogel (2010) 45--94.

\bibitem{FV11-1}
S. Friedl and S. Vidussi,
{\it Twisted Alexander polynomials detect fibered $3$-manifolds}, Ann. of Math.
{\bf 173} (2011), 1587--1643.

\bibitem{FV12-1}
S. Friedl and S. Vidussi,
{\it The Thurston norm and twisted Alexander polynomials},
arXiv:1204.6456.

\bibitem{GKM05-1}
H. Goda, T. Kitano and T. Morifuji,
{\it Reidemeister torsion, twisted Alexander polynomial
and fibered knots},
Comment. Math. Helv. {\bf 80} (2005), 51--61.

\bibitem{GM03-1}
H. Goda and T. Morifuji,
{\it Twisted Alexander polynomial for $SL(2,\C)$-representations
and fibered knots},
C. R. Math. Acad. Sci. Soc. R. Can. {\bf 25} (2003), 97--101.

\bibitem{GL84-1}
C. McA. Gordon and R. A. Litherland,
{\it Incompressible surfaces in branched coverings}
The Smith conjecture (New York, 1979), Pure Appl. Math., 112, Academic Press, Orlando, FL, (1984), 139--152.

\bibitem{HT85-1}
A. Hatcher and W. Thurston,
{\it Incompressible surfaces in $2$-bridge knot complements},
Invent. Math. {\bf 79} (1985), 225--246.

\bibitem{Herald97-1}
C. M. Herald,
{\it Existence of irreducible representations for knot complements with
nonconstant equivariant signature},
Math. Ann. {\bf 309} (1997), 21--35.

\bibitem{HK98-1}
M. Heusener and J. Kroll,
{\it Deforming abelian $SU(2)$-representations of knot groups},
Comment. Math. Helv. {\bf 73} (1998), 480--498.

\bibitem{HPS01-1}
M. Heusener, J. Porti and E. Su\'{a}rez Peir\'{o},
{\it Deformations of reducible representations of $3$-manifold groups into ${\rm SL}_2(\C)$},
J. Reine Angew. Math. {\bf 530} (2001), 191--227.

\bibitem{HLM95-1}
H. M. Hilden, M. T. Lozano and J. M. Montesinos-Amilibia,
{\it On the arithmetic $2$-bridge knots and link orbifolds and a new knot invariant},
J.~Knot Theory Ramifications {\bf 4} (1995), 81--114.

\bibitem{HLM03-1}
H. M. Hilden, M. T. Lozano and J. M. Montesinos-Amilibia,
{\it Character varieties and peripheral polynomials of a class of knots},
J.~Knot Theory Ramifications {\bf 12} (2003), 1093--1130.

\bibitem{IMS11-1}
M. Ishikawa, T. W. Mattman and K. Shimokawa,
{\it Tangle sums and factorization of A-polynomials},
arXiv:1107.2640.

\bibitem{Jaco80-1}
W. Jaco,
{\it Lectures on three-manifold topology},
CBMS Regional Conference Series in Mathematics, 43. American Mathematical Society,
Providence, R.I., 1980. xii+251 pp. ISBN: 0-8218-1693-4.

\bibitem{KM10-1}
T. Kim and T. Morifuji,
{\it Twisted Alexander polynomials and character varieties of $2$-bridge knot groups},
Internat. J. Math. {\bf 23} (2012), 1250022, 24pp.

\bibitem{Kitano11-1}
T. Kitano,
{\it Introduction to twisted Alexander polynomials of knots},
Lecture Note at Tokyo Inst. of Tech. (in Japanese), 2011. 92 pp.

\bibitem{KM05-1}
T. Kitano and T. Morifuji,
{\it Divisibility of twisted Alexander polynomials and fibered knots},
Ann. Sc. Norm. Super. Pisa Cl. Sci. (5)
{\bf 4} (2005) 179--186.

\bibitem{KS05-1}
T. Kitano and M. Suzuki,
{\it A partial order in the knot table},
Experiment. Math. {\bf 14} (2005), 385--390.
Corrigendum: Exp. Math. {\bf 20} (2011), 371.

\bibitem{Lin01-1}
X. S. Lin,
{\it Representations of knot groups and twisted Alexander polynomials},
Acta Math. Sin.
(Engl. Ser.) {\bf 17} (2001), 361--380.

\bibitem{LP85-1}
M. T. Lozano, J. H. Przytycki,
{\it Incompressible surfaces in the exterior of a closed $3$-braid. I. Surfaces with horizontal boundary components},
Math. Proc. Cambridge Philos. Soc. {\bf 98} (1985), 275--299.

\bibitem{Matsuda01-1}
H. Matsuda,
{\it Complements of hyperbolic knots of braid index four contain no closed
embedded totally geodesic surfaces},
Topology Appl. {\bf 119} (2001), 1--15.

\bibitem{Morifuji08-1}
T. Morifuji,
{\it Twisted Alexander polynomials of twist knots for nonabelian representations},
Bull. Sci. Math. {\bf 132} (2008), 439--453.

\bibitem{Morifuji12-1}
T. Morifuji,
{\it On a conjecture of Dunfield, Friedl and Jackson},
C. R. Acad. Sci. Paris, Ser. I {\bf 350} (2012), 921--924.

\bibitem{MT13-1}
T. Morifuji and A. T. Tran,
{\it Twisted Alexander polynomials of $2$-bridge knots for parabolic representations},
arXiv:1301.1101.

\bibitem{Oertel84-1}
U. Oertel,
{\it Closed incompressible surfaces in complements of star links},
Pacific J. Math. {\bf 111} (1984), 209--230.

\bibitem{ORS08-1}
T. Ohtsuki, R. Riley and M. Sakuma,
{\it Epimorphisms between $2$-bridge link groups},
in the Zieschang Gedenkschrift, Geom. and Topol. Monogr. {\bf 14} (2008), 417--450.

\bibitem{Shors91-1}
D. J. Shors,
{\it Deforming reducible representations of knot groups in
$SL_2(\C)$}, Thesis, U.C.L.A. 1991.

\bibitem{Wada94-1}
M. Wada,
{\it Twisted Alexander polynomial for finitely
presentable groups},
Topology {\bf 33} (1994), 241--256.
\end{thebibliography}
\end{document}